\documentclass[11pt]{amsart}
\usepackage{geometry}                
\usepackage{graphicx,comment}
\usepackage{amssymb,amscd}
\newtheorem{thm}{Theorem}

\newtheorem{cor}{Corollary}
\newtheorem{remark}{Remark}
\newtheorem{conj}{Conjecture}

\newtheorem*{NB*}{\textbf{N.B}}
\newtheorem{question}[thm]{Question}

\newtheorem{lemma}{Lemma}

\theoremstyle{definition}
\newtheorem{defn}{Definition}

\newcommand{\factor}[2]{{\raise0.7ex\hbox{$#1$} \!\mathord{\left/ {\vphantom {#1 {#2}}}\right.\kern-\nulldelimiterspace}
\!\lower0.7ex\hbox{${#2}$}}}
\newcommand{\BN}{\mathbb{N}}
\newcommand{\BZ}{\mathbb{Z}}
\newcommand{\BQ}{\mathbb{Q}}
\newcommand{\GG}{\ensuremath{\mathbb{G}}}

\newcommand{\mT}{\mathcal{T}}

\DeclareMathOperator{\gph}{gph}
\DeclareMathOperator{\St}{Star}

\DeclareMathOperator{\comp}{comp}
\DeclareMathOperator{\Aut}{Aut}

\title[Embeddability and quasi-isometric classification of pc groups]{Embeddability and quasi-isometric classification of partially commutative groups}
\author{Montserrat Casals-Ruiz}
\address{Departamento de Matemáticas, Universidad del País Vasco/Euskal Herriko Unibertsitatea, Barrio Sarriena s/n, Leioa, 48940 Bizkaia, Spain}
\email{montserrat.casals@ehu.eus}

\thanks{The author is supported by the Juan de la Cierva Programme and partly supported by the ERC grant PCG-336983, the Spanish Government grant MTM2014-53810-C2-2-P and Basque Government grant IT753-13.}


\begin{document}

\begin{abstract}

The main goal of this note is to suggest an algebraic approach to the  quasi-isometric classification of partially commutative groups (alias right-angled Artin groups). More precisely, we conjecture that if the partially commutative groups $\GG(\Delta)$ and $\GG(\Gamma)$ are quasi-isometric, then $\GG(\Delta)$ is a (nice) subgroup of $\GG(\Gamma)$ and vice-versa. We show that the conjecture holds for all known cases of quasi-isometric classification of partially commutative groups, namely for the classes of $n$-tress and atomic graphs.

As in the classical Mostow rigidity theory for irreducible lattices, we relate the quasi-isometric rigidity of the class of atomic partially commutative groups with the algebraic rigidity, that is with the co-Hopfian property of their $\BQ$-completions.
\end{abstract}
\maketitle
\section{Introduction}

A finitely generated group can be considered as a geometric object when endowed with a word metric. Gromov observed that this metric is in fact unique up to quasi-isometry and suggested the study of the rigidity problem, that is when a finitely generated group is quasi-isometric to a given one, and of the classification problem, i.e. when two groups (in a given class) are quasi-isometric.

In this paper, we focus on the question of  quasi-isometric classification of partially commutative groups (also known as right-anlged Artin groups) and its connection with the embeddability problem. Recall that a \emph{partially commutative group} (or a pc group, for short) is a finitely presented group $\GG(\Gamma)$ which can be described by a finite simplicial graph $\Gamma$, the commutation graph, in the following way: the vertices of $\Gamma$ are in bijective correspondence with the generators of $\GG(\Gamma)$ and the set of defining relations of $\GG(\Gamma)$ consists of  commutation relations, one for each pair of generators connected by an edge in $\Gamma$.

The quasi-isometric classification of pc groups has been previously considered by Behrstock, Januszkiewic and Neumann; Bestvina, Kleiner and Sageev and more recently by Huang. Although the results on classification are partial, they already exhibit a complex behaviour: on the one hand, some quasi-isometry classes contain many pc groups while others exhibit some type of rigidity. Furthermore, the techniques used to approach the problem in these cases do not seem to be amenable to address the general classification problem.

\smallskip

The main goal of this note is to suggest an algebraic approach towards the quasi-isometric classification of pc groups. More precisely we study (a stronger version of) the following question.

\begin{question}\label{ques:qi->emb}
If the pc groups $\GG(\Gamma)$ and $\GG(\Delta)$ are quasi-isometric, does this imply that $\GG(\Gamma) < \GG(\Delta)$ and $\GG(\Delta) < \GG(\Gamma)$?
\end{question}

Note that Question \ref{ques:qi->emb} has a positive answer if one strengthens the relation to commensurability, that is if $\GG(\Delta)$ and $\GG(\Gamma)$ are commensurable, then $\GG(\Delta) < \GG(\Gamma)$ and $\GG(\Gamma)<\GG(\Delta)$, see Lemma \ref{lem:commensurable}. Although this observation makes Question \ref{ques:qi->emb} natural, we will refine the statement and require, not only embeddability between the groups, but rather ``nice" embeddability in terms of the extension graphs. The extension graph $\Gamma^e$ of a graph $\Gamma$ was introduced in \cite{KK} to study the Embeddability Problem for pc groups, that is to determine when a pc group is a subgroup of another one. It is defined as follows: vertices of $\Gamma^e$ are in one-to-one correspondence with conjugates of generators of $\GG(\Gamma)$, that is 
$$
V(\Gamma^e)=\{ g^{-1} x g \in \GG(\Gamma) \mid x\in V(\Gamma), g\in \GG(\Gamma)\}
$$ 
and there is an edge in $\Gamma^e$ whenever the elements associated to the corresponding vertices commute in the group, i.e. 
$$
E(\Gamma^e)= \{ (u, v) \mid [u,v]=1 \hbox{ in } \GG(\Gamma)\}.
$$

In \cite{KK} it is shown that if $\Delta$ is an induced subgraph of the extension graph $\Gamma^e$, then $\GG(\Delta)$ embeds in $\GG(\Gamma)$ and that, under some additional conditions on the graphs, the converse also holds. These results suggested that the extension graph could be the graph-theoretical tool to determine when a pc group is a subgroup of another one. However, we show in \cite{CDK} that in general the extension graph is not enough to characterise embeddability: there are pc groups $\GG(\Delta)$ and $\GG(\Gamma)$ for which there exists an embedding from $\GG(\Delta)$ to $\GG(\Gamma)$ but $\Delta$ is not an induced subgraph of $\Gamma^e$. However, we believe that the extension graph may be helpful for the quasi-isometric classification of pc groups. More formally, we suggest the following conjecture.

\begin{conj}\label{conj:qi->EGE}
Let $\Delta$ and $\Gamma$ be simplicial graphs. If $\GG(\Delta)$ and $\GG(\Gamma)$ are quasi-isometric, then $\Delta < \Gamma^e$ and $\Gamma < \Delta^e$.
\end{conj}

If Conjecture \ref{conj:qi->EGE} holds, then we would get some algebraic control on the quasi-isometry classes of pc groups. For instance, it would follow that if $\GG(\Gamma)$ is a coherent pc group and $\GG(\Delta)$ is a pc group quasi-isometric to $\GG(\Gamma)$, then $\GG(\Delta)$ is also coherent, see Remark \ref{rem:coherent}. Recall that a group is coherent if all its finitely generated subgroups are finitely presentable. Furthermore, we would also gain an algorithmic understanding of the quasi-isometry classes of pc groups. Indeed, it was proven in \cite{C} that there is an algorithm that given two finite simplicial graphs $\Delta$ and $\Gamma$, decides whether or not $\Delta$ is an induced subgraph of the extension graph $\Gamma^e$.

\medskip

In this note we show that Conjecture \ref{conj:qi->EGE} holds for all known cases of quasi-isometric classification of pc groups. The two essential cases that need to be analysed are $n$-trees and atomic graphs.

\medskip

In \cite{BN}, the authors study the problem of classification of graph manifolds and proved that a pc group $\GG(\Delta)$ is quasi-isometric to $\GG(\Gamma)$, where $\Gamma$ is a tree of diameter greater than $2$, if and only if $\Delta$ is a tree of diameter greater than $2$.

This result was further generalised by Behrstock, Januszkiewic and Neumann for the class of $n$-trees (see Definition \ref{def:ntrees}). The authors prove that given two $n$-trees $\Delta$ and $\Gamma$, the corresponding pc groups $\GG(\Delta)$ and $\GG(\Gamma)$ are quasi-isometric if and only if the underlying trees associated to $\Delta$ and $\Gamma$ satisfy a graph-theoretic relation, namely they are bisimilar (see Definition \ref{def:bisimilar}).

In a different direction, Bestvina, Kleiner and Sageev, see \cite{BKS}, introduced and studied the problem of quasi-isometric classification of pc groups for atomic graphs, that is connected graphs with no valence 1 vertices, no cycles of length less than 5 and no separating closed stars of vertices. They prove that the class of pc groups defined by atomic graphs is quasi-isometrically rigid, that is given two atomic graphs $\Delta$ and $\Gamma$, the corresponding pc groups $\GG(\Delta)$ and $\GG(\Gamma)$ are quasi-isometric if and only if $\Delta$ and $\Gamma$ are isomorphic.

\medskip

Our goal is to prove that Conjecture \ref{conj:qi->EGE} holds in the aforementioned cases.

\begin{thm} 
Let $\mathcal{C}$ be one of the following classes of graphs:
\begin{itemize}
\item Triangle-built (i.e. graphs with no induced squares and no induced paths of length more than $2$);
\item Atomic graphs;
\item $n$-trees;
\end{itemize}
and let $\Delta, \Gamma \in \mathcal C$. Then $\GG(\Delta)$ is quasi-isometric to $\GG(\Gamma)$ if and only if $\Delta < \Gamma^e$ and $\Gamma < \Delta^e$.
\end{thm}

\begin{remark}
When this note was already written, Jingyin Huang published a new preprint \cite{Huang}, where he describes pc groups quasi-isometric to pc groups with finite outer automorphism group {\rm (}a class that naturally extends atomic pc groups{\rm )}. More precisely, Huang shows that if $\GG(\Gamma)$ is a pc group with finite outer automorphism group, then $\GG(\Delta)$ is quasi-isometric to $\GG(\Gamma)$ if and only if $\GG(\Gamma)$ and $\GG(\Delta)$ are commensurable if and only if $\Gamma^e$ and $\Delta^e$ are isomorphic. As a consequence, we have that {\rm Conjecture \ref{conj:qi->EGE}} also holds for pc groups with finite outer automorphism group.
\end{remark}

\bigskip

In many cases, there is a close relation between the group being quasi-isometrically and algebraically rigid, that is being co-Hopfian. For instance, in the classical case of irreducible lattices in semisimple Lie groups this relation is a consequence of Mostow rigidity.

As we mentioned above, the class of atomic pc groups is, in some sense, quasi-isometrically rigid so one can ask how far these groups are from being co-Hopfian. Recall that a group is called co-Hopfian if every injective endomorphism is an automorphism. In this direction, we study the set of injective endomorphisms for the class of atomic pc groups and show that ``up to taking roots", they are automorphisms. More precisely, we prove the following

\begin{cor}
Let $\Gamma$ be an atomic graph and $\psi: \GG(\Gamma) \to \GG(\Gamma)$ an injective endomorphism. Then, there exist $g\in \GG(\Gamma)$, $\sigma \in \Aut(\Gamma)$ and $k_v \in \BZ \setminus {0}$, $v\in V(\Gamma)$, such that
$$
\psi(v) = g^{-1}  \sigma(v)^{k_v} g, \ v\in V(\Gamma).
$$
In other words, up to conjugacy, graph automorphism and taking powers, $\psi$ is the identity endomorphism.
\end{cor}

\smallskip

In the spirit of classical theorems for abelian and locally nilpotent groups, we show in \cite{CDK2} that every pc group embeds into a divisible group, its $\BQ$-completion, which, roughly speaking, is the smallest divisible group containing $\GG(\Gamma)$. In other words, the $\BQ$-completion of $\GG(\Gamma)$ is an initial object in the category of divisible $\GG(\Gamma)$-groups (i.e. divisible groups with a designated copy of $\GG(\Gamma)$). From the description of the set of injective endomorphisms of an atomic pc group, one deduces that although atomic pc groups are not co-Hopfian, their $\BQ$-completions are. 

\begin{cor}
Let $\Gamma$ be an atomic graph. Then the $\BQ$-completion $\GG(\Gamma)^\BQ$ of $\GG(\Gamma)$ is co-Hopfian.
\end{cor}

\smallskip

In this context, it is natural to ask if for the class of pc groups the correspondence between quasi-isometric and algebraic rigidities holds in general. More precisely, we call a pc group $\GG(\Delta)$ weakly quasi-isometrically rigid if its quasi-isometry class is determined by the isomorphism type of its extension graph, that is $\GG(\Delta)$ is quasi-isometric to $\GG(\Gamma)$ if and only if the extension graph $\Delta^e$ is isomorphic to $\Gamma^e$. In this terminology, we ask

\begin{question}
Is it true that $\GG(\Gamma)$ is weakly quasi-isometrically rigid if and only if the $\BQ$-completion $\GG(\Gamma)^{\BQ}$ of $\GG(\Gamma)$ is co-Hopfian?
\end{question}

\medskip

We assume that the reader is familiar with basics of the theory of partially commutative groups. We refer the reader to \cite{C} and references there for preliminaries and notation.

\section{Elementary cases}

In this section we review some cases of quasi-isometric classification of pc groups, namely three classical cases: free, free abelian groups and direct product of two free groups; and the case of pc groups defined by trees and triangle-built graphs. 

\smallskip

In the case of free and free abelian groups, we have a complete classification: a finitely generated group $G$ is quasi-isometric to the free abelian group $\BZ^n$ if and only if it is virtually $\BZ^n$ (see \cite{Gromov}); and if $G$ is quasi-isometric to a (non-abelian) free group, then it is commensurable to it (and $G$ acts geometrically on a tree). When we restrict our consideration to the class of pc groups, we deduce that a pc group $\GG(\Delta)$ is quasi-isometric to $\BZ^n$ if and only if $\GG(\Delta)$ is isomorphic to $\BZ^n$; and a pc group $\GG(\Delta)$ is quasi-isometric to a non-abelian free group $F_n$ if and only if $\GG(\Delta)$ is a non-abelian free group $F_m$.

\smallskip

In the free abelian case, Conjecture \ref{conj:qi->EGE} holds trivially. Furthermore, the converse also holds. Indeed, the graph associated to $\BZ^n$ is a clique $\Gamma$ of dimension $n$. Hence, if $\Delta$ is an induced subgraph of the extension graph $\Gamma^e = \Gamma$, then $\Delta$ is a clique of dimension less than or equal to $n$. Furthermore, if $\Gamma < \Delta^e=\Delta$, then it follows that $\Delta=\Gamma$.

\smallskip

If $\GG(\Gamma)$ is a non-abelian free group, then the extension graph associated to $\Gamma$ is an infinite edgeless graph and so Conjecture \ref{conj:qi->EGE} holds. In this case, the converse is also true: if $\Delta < \Gamma^e$, then the graph of $\Delta$ is edgeless and so $\GG(\Delta)$ is a free group. If $\Gamma < \Delta^e$, it follows that $\Delta$ has at least two vertices and so $\GG(\Delta)$ is a non-abelian free group.

\smallskip

These results were generalised to groups acting on direct products of trees by several authors, see \cite{KKL,MSW, Ahl}: if a group $G$ is quasi-isometric to the direct product of two non-abelian free groups $F_n\times F_m$, then $G$ acts geometrically on the direct product of two trees. In particular, a pc group $\GG(\Delta)$ is quasi-isometric to $F_n\times F_m$  if and only if $\GG(\Delta)$ is the direct product of two non-abelian free groups $F_r \times F_s$.

The extension graph $\Gamma^e$ associated to $F_k \times F_l$ is the join graph of two infinite edgeless graphs and so it is immediate to check that Conjecture \ref{conj:qi->EGE} holds. The converse also holds: any induced subgraph $\Delta$ of $\Gamma^e$ is either edgeless or a join. If $\Gamma < \Delta^e$, it follows that $\Delta^e$ is the join of two infinite edgeless graphs and so $\GG(\Delta)$ is isomorphic to $F_{k'} \times F_{l'}$, $k',l'>1$.

\smallskip

We now turn our attention to the class of pc groups whose finitely generated subgroups are pc groups. In \cite{DromsPC}, Droms gives a graph-theoretic characterisation of this class: every finitely generated subgroup of $\GG(\Gamma)$ is a pc group if and only if $\Gamma$ is triangle-built, that is $\Gamma$ contains no induced squares and no induced paths of diameter more than $2$. In this case, Droms shows that $\GG(\Gamma)$ is isomorphic to $\BZ^n \times \GG(\Gamma')$, where $\Gamma'$ is the disjoint union of triangle-built graphs or, equivalently, $\GG(\Gamma')$ is the free product of triangle-built pc groups. By \cite{KKL}, we have that $\GG(\Delta)$ is quasi-isometric to $\GG(\Gamma)$ if and only if $\GG(\Delta)$ is isomorphic to $\BZ^n \times \GG(\Delta')$ and $\GG(\Gamma')$ and $\GG(\Delta')$ are quasi-isometric. Then, by \cite{P} it follows that each (one-ended) factor in the Grushko decomposition of $\GG(\Gamma')$ is equivalent to a (one-ended) factor in the Grushko decomposition of $\GG(\Delta)$ and vice-versa. By induction on the number of vertices in $\GG(\Gamma')$, we conclude that $\Delta' < {(\Gamma')}^e$ and  $\Gamma' < {(\Delta')}^e$ and so Conjecture \ref{conj:qi->EGE} also holds in this case.

\medskip

Let $T$ be a tree of diameter $2$, i.e.  $\GG(T)$ is isomorphic to $\BZ \times F_k$, $k>1$. Then $\GG(\Delta)$ is quasi-isometric to $\GG(T)$ if and only if $\GG(\Delta)$ is isomorphic to $\BZ \times F_n$, $n>1$. On the other hand, $\Delta < T^e$ and $T < \Delta^e$ if and only if $\GG(\Delta)$ is isomorphic to $\BZ \times F_n$, $n>1$. 

\medskip

The quasi-isometric classification for trees of diameter greater than or equal to 3 was established by Behrstock and Neumann. 

\begin{thm}[see \cite{BN}]
Let $T$ be a tree of diameter greater than or equal to 3. Then, $\GG(\Gamma)$ is quasi-isometric to $\GG(T)$ if and only if  $\Gamma$ is a tree of diameter greater than or equal to 3.
\end{thm}

It is easy to see, see for instance \cite{KK}, that any tree is an induced subgraph of the extension graph $T^e$ of a tree $T$ of diameter greater than or equal to 3.

On the other hand, since the extension graph of a tree is a tree, if $\Gamma$ is connected it follows that $\Gamma < T^e$ if and only if $\Gamma$ is a tree. If $T < \Gamma^e$, then we have that the diameter of $\Gamma$ is greater than or equal to $3$. In this case, we have shown that $\GG(\Gamma)$ is quasi-isometric to $\GG(T)$, where $T$ is a tree if and only if $\Gamma < T^e$ and $T< \Gamma^e$ and $\Gamma$ is connected.

\section{n-trees}

The class of $n$-trees was introduced and studied by Behrstock, Januszkiewic and Neumann in \cite{BJN}. We next recall some basic definitions and results and refer the reader to \cite{BJN} for further details.

\begin{defn}\label{def:ntrees}
We define the class of $n$-\emph{trees} $\mT_n$ to be the smallest class of $n$-dimensional simplicial complexes satisfying:
\begin{itemize}

\item the $n$-simplex is in $\mT_n$;
\item If $K_1$ and $K_2$ are complexes in $\mT_n$ then the union of $K_1$ and $K_2$ along any
$(n-1)$-simplex is in $\mT_n$.
\end{itemize}
\end{defn}

For $n = 1$ this is the class of finite trees. In this section, we consider pc groups defined by $n$-trees. Although formally we should define these pc groups by the 1-skeleton of the $n$-tree, we abuse the notation and write $\GG(K)$, where $K\in \mT$.

Fix a complex $K \in  \mT_n$. We define a \emph{piece} to be the star in $K$ of an $(n-1)$-simplex of $K$ which is the boundary of at least two $n$-simplices. Let $P$ denote a piece of $K$. Then, $P$ consists of a finite collection of $n$-simplices attached along the common $(n-1)$-simplex, i.e. the join of the $(n-1)$-simplex with a finite set of points $p_1,\dots,p_k$.

\begin{defn}
To each $K \in  \mT_n$ we associate a labelled bipartite tree $\gph(K)$ as follows. To each piece in $K$ we assign a vertex labelled $p$ (for piece). To each of the $n$-simplices which is in more than one piece we assign a vertex labelled $f$ (for face). Each $f$-vertex is connected by an edge to each of the $p$-vertices which corresponds to a piece containing the $n$-simplex.

Since for any $K \in \mT_n$ there is a simplicial map to an $n$-dimensional simplex $\Delta$, which is unique up to permutation of $\Delta$, it follows that labelling the vertices of $\Delta$ by $1$ to $n + 1$ pulls back to a consistent labelling on all the vertices of $K$. Note that in any piece all the vertices of their common $(n-1)$-simplex (the ``spine'' of the piece) are given the same label. We label each $p$-vertex by the index of the $n$-simplex vertex which is not on the spine of the corresponding piece. Hence the label set for the $p$-vertices are the elements of the set $\{1,\dots, n + 1\}$. The possible labels for vertices are thus $p1, p2, \dots, p{(n+1)}$ and $f$, for a total of $n+2$ possible labels.
The $p/f$-distinction gives a bipartite structure on our tree $\gph(K)$. The $p$-vertices to which a given $f$-vertex is connected have distinct labels, so a $f$-vertex has valence at most $n + 1$ (and at least 2). A $p$-vertex can be connected to any number of $f$-vertices.
\end{defn}

Note that the graph $\gph(K)$ associated to an $n$-tree corresponds to the graph of the graph-of-groups decomposition of $\GG(K)$ where vertex groups are fundamental groups of the pieces and edges groups are labelled by the fundamental group of an $n$-simplex (i.e. edge groups are free abelian).

\begin{defn}
A \emph{coloured graph} is a graph $\Gamma$, a set $C$, and a ``vertex colouring'' $c: V (\Gamma) \to C$.
A \emph{weak covering of coloured graphs} is a graph homomorphism $f : \Gamma \to \Gamma'$ which respects colours and has the property: for each $v \in V(\Gamma)$ and for each edge $e' \in E(\Gamma')$ at $f(v)$ there exists an $e \in E(\Gamma)$ at $v$ with $f(e) = e'$.
\end{defn}

Henceforth, we assume all graphs we consider to be connected. It is easy to see that a weak covering is then surjective.

\begin{defn}\label{def:bisimilar}
Coloured graphs  $\Gamma_1$, $\Gamma_2$ are \emph{bisimilar} if $\Gamma_1$ and $\Gamma_2$ weakly cover some common coloured graph.
\end{defn} 

The main result of \cite{BJN} is the quasi-isometric classification of $n$-trees in terms of bisimilarity of the defining graphs. More precisely, the authors prove:

\begin{thm}[Theorem 1.1, \cite{BJN}]

Given $K, K' \in \mT_n$. The groups $\GG(K)$ and $\GG(K')$ are quasi-isometric if and only if $\gph(K)$ and $\gph(K')$ are bisimilar after possibly reordering the $p$-colours by an element of the symmetric group on $n + 1$ elements.
\end{thm}

\smallskip

The goal of this section is to show that the graphs associated to $n$-trees $\Delta$ and $\Gamma$ are bisimilar (after possibly reordering the $p$-colours) if and only if $\Delta <\Gamma^e$ and $\Gamma < \Delta^e$.

\smallskip

Before we turn our attention to the proof, we recall an easy but very useful lemma that describes the extension graph $\Gamma^e$ as a sequence of ``doublings" over stars of vertices.

\begin{lemma}[Lemma 22, \cite{KK}]\label{lem:22kk}
Let $\Gamma$ be a finite graph and $\Delta$ be a finite induced subgraph of $\Gamma^e$. Then there exists an $l > 0$, a sequence of vertices $v_1,v_2,...,v_l$ of $\Gamma^e$, and a sequence of finite induced subgraphs $\Gamma$ = $\Gamma_0 \le \Gamma_1 \le \dots \le \Gamma_l$ of $\Gamma^e$ where $\Gamma_i$ is obtained by taking the double of $\Gamma_{i-1}$ along $\St_{\Gamma_{i-1}}(v_i)$ for each $i = 1,\dots,l$, such that $\Delta \le \Gamma_l$.
\end{lemma}

\begin{lemma}\label{lem:weakcoverimplege}
Assume that $\gph(\Delta)$ weakly covers $\gph(\Gamma)$, then $\Delta <\Gamma^e$.
\end{lemma}

\begin{proof}
Assume that $f:\gph(\Delta) \to \gph(\Gamma)$ is a weak covering. Since both graphs are connected and without multiple edges, $f$ is surjective and so the graph $\gph(\Gamma)$ is an induced subgraph of $\gph(\Delta)$.

We prove the statement by induction on the number $m=|V(\gph(\Delta)) \setminus V(\gph(\Gamma))|$.

\textit{Base of induction: let $m=0$.} In this case, the $\gph(\Gamma)$ is an induced subgraph with the same set of vertices and so it coincides with $\gph(\Delta)$. It follows from the definition of $\gph(\Delta)$ and $\gph(\Gamma)$ that $\Delta$ and $\Gamma$ only differ by their ``leaves'', that is they differ by the set of $n$-simplices such that only one of its $(n-1)$-faces is a face of another $n$-simplex. 

Define the \emph{core} of an $n$-tree $\Lambda$ to be the $n$-subtree of $\Lambda$ that consists of $n$-simplices which are not leaves of $\Lambda$. In this terminology, if $m=0$, then the cores of $\Delta$ and $\Gamma$ are isomorphic (possibly empty). We denote by $\varphi'$ a colour preserving isomorphism from the core of $\Delta$ to the core of $\Gamma$. 

Our goal is to define an embedding $\varphi$ from $\Delta$ to $\Gamma^e$ that extends the isomorphism $\varphi'$. Without loss of generality, assume that $\Delta$ and $\Gamma$ are different, that is there exists an $(n-1)$-simplex $F$ which is a face of at least two $n$-simplices, one of these simplices is a leaf and the pieces associated to $F$ in $\Delta$ and in $\Gamma$ differ in the number of leaves. For the $n$-simplices that contain the face $F$ in $\Gamma$ (resp. in $\Delta$) and which belong to the core, we denote by $x_1, \dots, x_r$ (resp. $x_1', \dots, x_r'$) the vertices of these simplices that do not belong to the face $F$ in $\Gamma$ (resp. in $\Delta$). For the other $n$-simplices, that is the leaves that contain $F$, we denote by $y_1, \dots, y_k$ (resp. $y_1', \dots, y_l'$, $l >k>0$) the vertices that do not belong to the face $F$ in $\Gamma$ (resp. in $\Delta$). Note that $r$ may be $0$, in which case $k\ge 2$ and $\gph(\Gamma)$ is a vertex.

We define an embedding from the core of $\Delta$ together with the piece associated to $F$ to $\Gamma^e$ as follows: $\varphi$ on the core of $\Delta$ is defined as $\varphi'$ and so in particular, $\varphi(x_i')=x_i$, $i=1, \dots, r$; and $\varphi(y_i')=y_1^{(x_1^i)}$, $i=1, \dots, l$ if $r\ne 0$ and otherwise $\varphi(y_i')=y_1^{(y_2^i)}$. It is easy to check that the map $\varphi$ induces a graph embedding, since $y_i$ and $x_j$ do not commute with each other, $i=1, \dots, k$, $j=1, \dots,r$, but they commute with the vertices in the face $F$.

Repeating the above argument for the pieces for which the number of leaves in $\Delta$ is different to that of $\Gamma$, one obtains an embedding from $\Delta$ to $\Gamma^e$. Note that by construction, vertices of a piece $P$ (resp. face) in $\Delta$ which is identified to a vertex $v$ in $\gph(\Delta)$ are sent to conjugates of vertices of the piece $P'$ (resp. face) in $\Gamma$ which is identified to the same vertex $v$ in $\gph(\Gamma)$. Furthermore, vertices of faces in $\Delta$ identified with vertices of $\gph(\Delta)$ are sent to the same conjugate (in this case, the conjugating element is trivial).

\smallskip

\textit{Induction step.} Let $v$ be a leaf of $\gph(\Delta)$ such that the weak covering $f$ restricted to $\gph(\Delta) \setminus \{v\}$ is again a weak covering of graphs. Note that such $v$ exists because $\gph(\Delta)$ and $\gph(\Gamma)$ are trees and $|V(\Delta)| \gneq |V(\Gamma)|$. Assume by induction that there is an embedding $\psi$ from the $n$-tree $\Delta'$ associated to $\gph(\Delta) \setminus \{v\}$ into $\Gamma^e$ such that 
\begin{itemize}
\item vertices of a piece $P$ in $\Delta'$ which is identified to a $p$-vertex $v$ in $\gph(\Delta')$ are sent to conjugates of vertices of the piece $P'$ in $\Gamma$ which is identified to the vertex $f(v)$ in $\gph(\Gamma)$; and
\item vertices of a face $F$ in $\Delta'$ which is identified to an $f$-vertex $w$ in $\gph(\Delta')$ are sent to the \emph{same} conjugate of vertices of the face $F'$ in $\Gamma$ which is identified to the $f$-vertex $f(w)$ in $\gph(\Gamma)$. 
\end{itemize}

By Lemma \ref{lem:22kk}, $\Delta'$ embeds into $\Gamma_{n-1} < \Gamma^e$, where $\Gamma_{n-1}$ is obtained from $\Gamma$ by a sequence of doublings.

Let $w$ be the vertex of $\gph(\Delta)$ so that $(v,w) \in E(\gph(\Delta))$. Note that from the construction of $\gph(\Delta)$, it follows that $v$ is a $p$-vertex and $w$ is an $f$-vertex. By induction the embedding of $\Delta'$ into $\Gamma_{n-1}$ satisfies that the image under $\psi$ of vertices $a_i$ in $\Delta$ associated to the face identified with $w$ is $b_{j_i}^{g_{n-1}}$, for some $g_{n-1}\in \GG(\Gamma)$ and $b_{j_i}$ a vertex in the face identified with $f(w)$. Let $h_n \in \GG(\Gamma)$ be so that the alphabet of $h_n$ is exactly the alphabet that labels the face $F'$ identified with $f(w)$. We can choose $h_n$ so that the element $g_n= h_n^{g_{n-1}^{-1}}$ has not appeared in the (finite) sequence of doubling used to construct $\Gamma_{n-1}$. Indeed, there are at most $n-1$ doublings and the subgroup associated to faces are free abelian groups of rank at least 2. Conjugating by $g_n$, we obtain a doubling of $\Gamma_{n-1}$ along $F'^{g_{n-1}} = F'^{ g_{n-1} g_n}$, that is $\Gamma_n = \Gamma_{n-1} \bigsqcup\limits_{F'^{g_n}} \Gamma_{n-1}^{g_n}$. 

If the piece identified with $v$ in $\gph(\Delta)$ is the same as the piece identified with $f(v)$ in $\Gamma$, then the embedding $\psi$ of $\Delta'$ into $\Gamma^e$ can be extended to an embedding of $\Delta$. Indeed, it suffices to send the vertices $a_i$ of the $n$-simplices of the piece that are not in the face $F$ to $(b_i^{g_{n-1}})^{g_n}$, where $b_i$ are the corresponding vertices in the piece in $\Gamma$ that are not in the face identified with $f(w)$. 

If the pieces have a different number of leaves, we define the embedding $\psi$ as in the base of induction, that is $\psi$ sends the vertices $a_i$ of the $n$-simplices of the piece that are not in the face $F$ to $(\varphi(b_i)^{g_{n-1}})^{g_n}$, where $b_i$ are the corresponding vertices in the piece in $\Gamma$ that are not in the face identified with $f(w)$. By construction the embedding $\psi$ satisfies the induction hypothesis. 
\end{proof}

\begin{cor}
Let $\Gamma, \Delta \in \mathcal T_n$. If $\gph(\Delta)$ and $\gph(\Gamma)$ are bisimilar, then $\Delta <\Gamma^e$ and $\Gamma < \Delta^e$.
\end{cor}
\begin{proof}
If $\gph(\Delta)$ and $\gph(\Gamma)$ are bisimilar, it follows from the definition that they weakly cover a graph $\gph(\Lambda)$. Since the graphs are assumed to be connected, it follows that the weak covering is an epimorphism and so, in particular, there is an embedding of $\gph(\Lambda)$ into $\gph(\Gamma)$ and into $\gph(\Delta)$, hence $\Lambda< \Gamma$ and $\Lambda < \Delta$ and so $\Lambda^e < \Gamma^e$ and $\Lambda^e< \Delta^e$. It follows from Lemma \ref{lem:weakcoverimplege}, that $\Delta < \Lambda^e < \Gamma^e$ and $\Gamma < \Lambda^e < \Delta^e$. 
\end{proof}

\begin{lemma}
Let $\Gamma, \Delta \in \mathcal T_n$. Then if $\Delta < \Gamma^e$ and $\Gamma < \Delta^e$, then $\gph(\Delta)$ and $\gph(\Gamma)$ are bisimilar.
\end{lemma}
\begin{proof}
Let us first show that one can assume $\Delta$ to be minimal, that is for any proper subgraph $\Delta'$ of $\Delta$ such that $\Delta' \in \mathcal T_n$, the graph $\gph(\Delta)$ is not bisimilar to $\gph(\Delta)'$. Let $\Delta'$ be a subgraph of $\Delta$ such that $\Delta'\in \mathcal T_n$ and the graphs $\gph(\Delta')$ and $\gph(\Delta)$ are bisimilar. If $\Delta < \Gamma^e$ and $\Gamma < \Delta^e$, then $\Delta' < \Gamma^e$ and $\Gamma < {\Delta'}^e$. Indeed, $\Delta'< \Delta$ and by assumption, $\Delta < \Gamma^e$, hence $\Delta' < \Gamma^e$; on the other hand, by assumption $\Gamma < \Delta^e$ and since $\gph(\Delta)$ and $\gph(\Delta')$ are bisimilar, by Lemma \ref{lem:weakcoverimplege} we have that $\Delta^e < {\Delta'}^e$ and so $\Gamma < {\Delta'}^e$. Furthermore, if the statement holds for $\Delta'$ and $\Gamma$, that is $\gph(\Delta')$ and $\gph(\Gamma)$ are bisimilar, then we conclude that the statement also holds for $\Delta$ and $\Gamma$, since $\gph(\Delta)$ is bisimilar to $\gph(\Delta')$ and by transitivity of the relation, $\gph(\Delta)$ and $\gph(\Delta')$ are bisimilar.

We further assume that $\Delta$ is minimal.

 \smallskip

Observe that if $\Gamma$, $\Gamma'$ and $\Lambda$ are $n$-trees and $\Lambda<\Gamma$, $\Lambda<\Gamma'$, then $\Gamma\bigsqcup\limits_{\Lambda} \Gamma'$ is also an $n$-tree. Furthermore, 
$$
\gph(\Gamma\bigsqcup\limits_{\Lambda} \Gamma')= \gph(\Gamma) \bigsqcup\limits_{\gph(\Lambda)} \gph(\Gamma').
$$

By Lemma \ref{lem:22kk}, we have that if $\Delta$ is a finite induced subgraph of $\Gamma^e$, then there exist $l > 0$, a sequence of vertices $v_1,v_2,...,v_l$ of $\Gamma^e$, and a sequence of finite induced subgraphs $\Gamma$ = $\Gamma_0 \le \Gamma_1 \le \dots \le \Gamma_l$ of $\Gamma^e$, where $\Gamma_i$ is obtained by taking the double of $\Gamma_{i-1}$ along $\St_{\Gamma_{i-1}}(v_i)$ for each $i = 1,\dots,l$, such that $\Delta \le \Gamma_l$.

Note that the star of a vertex of an $n$-tree $\Gamma$ is an $n$-tree. If we assume by induction that $\Gamma_{l-1}$ is an $n$-tree (and so is $\St_{\Gamma_{l-1}}(v)$ for every vertex in $\Gamma_{l-1}$), then it follows from the above observation that the double $\Gamma_l$ over the $n$-tree $\St_{\Gamma_{l-1}}(v_{l-1})$ is again an $n$-tree. Furthermore, the double of a tree $T$ over a subtree $T'$, i.e. $T \bigsqcup\limits_{T'} T$ is bisimilar to $T$ and so by induction, $\gph(\Gamma_l)$ is bisimilar to $\gph(\Gamma)$.

Since $\Delta < \Gamma_l$ and $\Delta$ is an $n$-tree, it follows that $\gph(\Delta)$ is a subtree of $\gph(\Gamma_l)$. Since $\gph(\Gamma_l)$ is bisimilar to $\gph(\Gamma)$, it follows that $\gph(\Delta)$ is bisimilar to a subgraph of $\gph(\Gamma)$. 

Since by assumption $\Gamma < \Delta^e$, it follows that either $\Gamma < \Delta$ (and so $\gph(\Gamma) < \gph(\Delta)$) or the above argument applies and so $\gph(\Gamma)$ is bisimilar to a subgraph of $\gph(\Delta)$.

Since $\gph(\Delta)$ is bisimilar to a subgraph of $\gph(\Gamma)$ and $\gph(\Gamma)$ is in turn bisimilar to a subgraph of $\gph(\Delta)$, it follows that $\Delta$  is bisimilar to a subgraph of itself. Since by assumption $\Delta$ is minimal (i.e. not bisimilar to any proper subgraph), we conclude that the subgraph is $\Delta$ and so $\Delta < \Gamma$ and $\Gamma$ is bisimilar to $\Delta$.
\end{proof}

We summarise the results of this section in the following corollary.

\begin{cor}
Let $\Gamma, \Delta \in \mathcal T_n$. Then $\GG(\Gamma)$ and $\GG(\Delta)$ are quasi-isometric if and only if $\gph(\Delta)$ and $\gph(\Gamma)$ are bisimilar if and only if $\Delta < \Gamma^e$ and $\Gamma < \Delta^e$.
\end{cor}

\section{Atomic graphs}

In this section we study the class of atomic pc groups introduced by Bestvina, Kleiner and Sageev in \cite{BKS}. Recall that an atomic graph is a graph with no valence 1 vertices, no cycles of length less than 5 and no separating closed stars of vertices. More precisely, in this section we center in the algebraic rigidity of the class of atomic pc groups and show that if an atomic pc group $\GG(\Delta)$ embeds into an atomic pc group $\GG(\Gamma)$ and vice-versa, then the groups $\GG(\Delta)$ and $\GG(\Gamma)$ are isomorphic. 

\smallskip

In order to study the group embeddability into an atomic pc group $\GG(\Gamma)$, it suffices to study the graph embeddability into the extension graph $\Gamma^e$, see \cite{KK}. We begin by recalling the following technical lemma, which follows from the proof of \cite[Lemma 26(6)]{KK} and summarises the tree-like properties of the extension graph $\Gamma^e$.

\begin{lemma}[see Lemma 26 in \cite{KK}] \label{lem:26}
Let $\Gamma$ be an atomic graph and let $\Gamma^g$ and $\Gamma^h$ be two subgraphs of the extension graph $\Gamma^e$. Then
\begin{enumerate}
\item the intersection of $\Gamma^g$ and $\Gamma^h$ is either empty or is contained in the star $\St_{\Gamma^e}(v)$ of some vertex  $v\in \Gamma^g\cap \Gamma^h$ and $\Gamma^g\cap \Gamma^h=\St_{\Gamma^g}(v)= \St_{\Gamma^h}(v)$;
\item the star of any vertex disconnects the extension graph $\Gamma^e$, moreover any two vertices $v$ and $w$ which do not belong to the same conjugate of $\Gamma$ in $\Gamma^e$ can be separated by the star of some vertex of $\Gamma^e$;
\item let $\Gamma^g$ and $\Gamma^h$ be non trivial, $\Gamma^g\cap \Gamma^h=\St_{\Gamma^g}(v)= \St_{\Gamma^h}(v)$, then $\Gamma^g\cap \Gamma^h$ disconnects the extension graph $\Gamma^e$;
\item let $p=(v_1,\dots, v_k, w_1,\dots, w_l, v_{k+1}, \dots, v_m)$ be a simple path in $\Gamma^e$ so that $v_i\in \Gamma$ and $w_j\in \Gamma^e\smallsetminus \Gamma$, $i=1,\dots, m$, $j=1,\dots, l$, then there exists $u\in \Gamma$ so that $v_{k}, v_{k+1}\in \St_{\Gamma}(u)$.
\end{enumerate}
\end{lemma}

We next show that the embeddability of an atomic graph $\Gamma$ into the corresponding extension graph $\Gamma^e$ is rigid, that is, any embedding $\varphi: \Gamma \to \Gamma^e$ is the identity up to an automorphism. 

\begin{thm}\label{lem:atomicembedding}
Let $\Gamma$ be an atomic graph and $\varphi: \Gamma \to \Gamma^e$ an embedding of $\Gamma$ into the extension graph $\Gamma^e$ {\rm (}as a full subgraph{\rm )}. Then there exists $g\in \GG(\Gamma)$ and an automorphism $\alpha$ of $\Gamma$ so that $\varphi(\alpha(\Gamma))=\Gamma^g$, i.e. up to conjugacy and graph automorphism there is only one way to embed $\Gamma$ into $\Gamma^e$.
\end{thm}
\begin{proof}

Let $\Gamma$ be an atomic graph and let $n$ be the rank of $\pi_1(\Gamma)$ ($\pi_1(\Gamma)$ is isomorphic to the free group $F_n$ of rank $n$). Recall that since $\Gamma$ is atomic, there are no vertices of valence one. We mark $n$ cycles in $\Gamma$ as follows. If $n=1$, then $\pi_1(\Gamma)$ is isomorphic to $\BZ$ and we mark the only cycle in $\Gamma$. Let $T'$ be a maximal subtree of $\Gamma$ and $\Gamma \setminus T'$ be the set of edges $\{e_1, \dots, e_n\}$. By the length of a cycle $c$, denoted by $|c|$, we mean the length of its core. As usual, if we fix a base point in $T'$, each edge $e_i$ defines a cycle $c_{T', e_i}$ in $\Gamma$, $i=1,\dots, n$. Note that by definition, the length of the cycle defined by an edge $e_i$ is independent of the choice of the base point.

Re-enumerating if necessary, we shall assume that $|c_{T',e_i}| \leq |c_{T',e_j}|$, $1\le i<j\le n$. Every maximal subtree $T'$ of $\Gamma$ defines a tuple $(|c_{T',e_1}|, \dots, |c_{T',e_n}|)$ (independent of the choice of the base point). The natural lexicographical order on the tuples $(|c_{T',e_1}|, \dots, |c_{T',e_n}|)$, induces an order on the set of maximal subtrees of $\Gamma$, namely $T < T'$ if $(|c_{T,e_1}|, \dots, |c_{T,e_n}|)$ is less than $(|c_{T',e_1}|, \dots, |c_{T',e_n}|)$ in the lexicographical order, that is there is $i\in \{1, \dots, n\}$ such that $|c_{T,e_k}| = |c_{T',e_k}|$ for $1\le k < i$ and $|c_{T,e_i}| < |c_{T',e_i}|$. Let $T$ be a minimal (in the above order) maximal subtree of $\Gamma$, i.e. $T \leq T'$ for all maximal subtrees $T'$ of $\Gamma$. We mark the cores of the cycles $c_{e_1,T}, \dots, c_{e_n,T}$. Until the end of the proof of this lemma, unless stated otherwise, by a cycle in $\Gamma$ we mean the core of one of the cycles $\{c_{e_1,T}, \dots, c_{e_n,T}\}$. Note that since the graph $\Gamma$ has no vertices of valence one and since $T$ is a maximal subtree of $\Gamma$, every edge of $\Gamma$ belongs to (the core of at least one of) the cycles $c_{e_i,T}$.

\medskip

Note that two different conjugates of $\Gamma$, say $\Gamma^g$ and $\Gamma^h$, share at most the star of a vertex $u_{g,h}$, that is $u_{g,h} \in V(\Gamma^g)\cap V(\Gamma^h)$ and $\Gamma^g \cap \Gamma^h \subset \St(u_{g,h})$, see Lemma \ref{lem:26}. 

\textit{Claim 1.} If the vertex $u_{g,h}$ belongs to $\varphi(\Gamma)$, then either $\varphi(\Gamma) \cap (\Gamma^g \setminus \Gamma^h) = \emptyset$ or $\varphi(\Gamma) \cap (\Gamma^h \setminus \Gamma^h) = \emptyset$.

Indeed, if $\varphi(\Gamma)$ intersects $\Gamma^g$ and $\Gamma^h$ outside the star $\St(u_{g,h})$, then by Lemma \ref{lem:26}, $\St(u_{g,h})$ separates the extension graph $\Gamma^e$ (and $\Gamma^g$ and $\Gamma^h$ belong to different connected components), so $\St(u_{g,h})$ also separates the image $\varphi(\Gamma) \simeq \Gamma$. However, we assume $\Gamma$ to be atomic and so in particular it does not contain vertices with closed separating stars.

\textit{Claim 2.} Let $c$ be a cycle in $\Gamma$ and assume that $\varphi(c) < \Gamma^e$ is not contained in one conjugate of $\Gamma$, that is $\varphi(c) \cap \Gamma^g \ne \varphi(c)$, for all $g\in \GG(\Gamma)$. Then $\varphi(c)$ is contained in the union of cycles $c_i^{g_i}< \Gamma^{g_i}$, $g_i \in \GG(\Gamma)$ and the length of each $c_i$ is strictly less than the length of $\varphi(c)$.

Indeed, since $\varphi(c)$ is not contained in one conjugate of $\Gamma$, it follows that there exist $v,w \in \varphi(c)$ and $g,h \in \GG(\Gamma)$ such that $v \in \Gamma^g \setminus \Gamma^h$, $w \in \Gamma^h \setminus \Gamma^g$ and $\Gamma^g \cap \Gamma^h\ne \emptyset$. Since  $\Gamma^g \cap \Gamma^h\ne \emptyset$, it follows from Lemma \ref{lem:26} that there exists $u_{g,h}$ such that its star $\St(u_{g,h})$ separates $\Gamma^g$ and $\Gamma^h$ and so separates $\varphi(c)$. Furthermore, from Claim 1 it follows that $u_{g,h} \notin \varphi(c)$. Let $d_1, \dots, d_k$, $k\ge 2$, be the connected components of $\varphi(c) \setminus \St(u_{g,h})$. For each $d_i$, there is a path $p_i$ in the star of $u_{g,h}$ of length at most 2 such that $d_i \cup p_i$ is a cycle $c_i$ in $\Gamma^e$. Since $\Gamma$ is atomic, it follows that there are no squares or triangles in $\Gamma^e$ and so we have that $|c_i|=|d_i\cup p_i| \ge 5$. Furthermore, since $|p_i|\le 2$, it follows that $|d_i| \ge 3$. We conclude that $|c_i| = |d_i|+|p_i| \le |d_i|+2 < |d_1| + \dots + |d_k| = |\varphi(c)| = |c|$, since $k\ge 2$ and $|d_i| \ge 3$.

\textit{Claim 3.} If $\Gamma'$ is a full subgraph of $\Gamma$, $c$ is a cycle of minimal length in $\Gamma$, $\Gamma'$ and $c$ intersect at least in an edge $e$ and $\varphi(\Gamma')$ is contained in a conjugate of $\Gamma$, say $\varphi(\Gamma') <\Gamma^g$, then $\varphi(c)$ is also contained in $\Gamma^g$, i.e. $\varphi(\Gamma'\cup c) < \Gamma^g$. 

Indeed, since $c$ is of minimal length, it follows from Claim 2 that $\varphi(c)$ is contained in one conjugate of $\Gamma$. Assume towards contradiction that $\varphi(c) < \Gamma^h$ and $\Gamma^h \ne \Gamma^g$. Since $\Gamma'$ and $c$ share at least the edge $e$, it follows that $\varphi(e) \in \Gamma^g \cap \Gamma^h$. Since conjugates of $\Gamma$ share at most a star of a vertex, it follows that $u_{g,h}$ is a vertex of $e$ and so belongs to $\varphi(\Gamma)$, contradicting Claim 1. Hence $\varphi(\Gamma' \cup c) < \Gamma^g$, for some $g\in \GG(\Gamma)$.

\medskip

Let $c$ be a cycle in $\Gamma$. Define the neighbourhood $N(c)$ of the cycle $c$ to be the collection of all cycles in $\Gamma$ that share at least one edge with $c$. Recall that by cycles in $\Gamma$ we mean marked cycles. We say that a connected subgraph $S$ of $\Gamma$ is a \emph{component} if it is a union of cycles and has no cut-points. The neighbourhood $N(S)$ of a component $S$ is the union of neighbourhoods of all the cycles which belong to $S$. Note that, by definition, the neighbourhood of a component is itself a component. 

Let $S$ be a component. Given a cycle $c$ in $S$, the complexity $\comp(c)$ of $c$ in $S$ is
$$
\comp(c)=(r_5, \dots, r_M) \in \BN^{M-4},
$$
where $r_l$ is the number of cycles of length $l$ that belong to $N(c)\cap S$, $l=5, \dots M$ and $M=|c_{T,e_n}|$. Define the finite set of complexities of cycles in $S$ as
$$
\mathcal K(S) = \{ \comp(c) \mid \hbox{ c is a cycle in S} \}
$$
and $\mathcal K$ to be the union of complexities $\mathcal K(S)$ over all components $S$.

To a component $S$, we associate the tuple $(m_{l,\comp_i})_{5\le l \le M, \comp_i \in \mathcal K}^S$, where $m_{l,\comp_i}$ is the number of cycles in $S$ of length $l$ and complexity $\comp_i$ (in $S$) ordered lexicographically from minimal to maximal length and from maximal to minimal complexity, that is the position $(k, \comp_i)$ in the tuple is before the position $(k',\comp_i')$ if either $k<k'$ or $k=k'$ and $\comp_i > \comp_i'$. If no confusion arises, we drop the subindices and denote the tuple simpy by $(m_{l,\comp_i})^S$.

The lexicographical order on the tuples $(m_{l,\comp_i})^S$ naturally defines an ordering $\prec$ on the components of $\Gamma$: $S'\prec S$ if an only if $(m_{l,\comp_i})^S < (m_{l,\comp_i})^{S'}$ in the lexicographical order, that is there exists $(l, \comp_j)$ such that $m_{k,\comp_i}(S)=m_{k,\comp_i}(S')$ for all $(k,\comp_i) < (l, \comp_j)$ and $m_{l,\comp_j}(S') < m_{l,\comp_j}(S)$. In this ordering a component $S$ is maximal if it contains the maximal number of cycles of minimal length and maximal complexity. 

\smallskip

Let $c$ be a cycle of minimal length in $\Gamma$. Define the \emph{minimal component} of $c$ to be the maximal connected subgraph $C$ of $\Gamma$ containing $c$ so that $C$ is a union of cycles of minimal length in $\Gamma$ without cut points. Let $\mathcal S_{1}$ be the set of minimal components of cycles of minimal length in $\Gamma$ which are maximal with respect to the order $\prec$. 

We define the components $\mathcal S_q$ recursively as follows. Consider the set $\mathcal T_q= \{N(S_{q-1})\mid S_{q-1}\in \mathcal S_{q-1}\}$ of neighbourhoods of components of $\mathcal S_{q-1}$. Define $\mathcal S_q$ to be the set of maximal components of $\mathcal T_q$ with respect to the order $\prec$. Note that for $q$ large enough we have that $S_q=\Gamma$, for all $S_q\in \mathcal{S}_q$. We prove by induction on $q$ that for all $S_q \in \mathcal S_q$, $\varphi(S_q) < \Gamma^g$, for some $g\in \GG(\Gamma)$.

\textit{Base of induction.} Let $S_1 \in \mathcal S_1$. By definition of $S_1$, recursively applying Claim 3, it follows that $\varphi(S_1)$ is contained in $\Gamma^g$ for some $g\in \GG(\Gamma)$. Since $\varphi$ is a monomorphism of graphs, $\varphi(S_1)$ is a component of $\Gamma^g$. Since $S_1$ is maximal in the sense of $\prec$, since $\varphi(S_1)\subseteq \Gamma^g\simeq \Gamma$ and since $\varphi$ is an embedding, it follows that the complexity of the image $\comp(\varphi(c))$ in $\varphi(S_1)$ is equal to the complexity $\comp(c)$ in $S_1$ and $\varphi(S_1)\in \mathcal S_1 (\Gamma^g)$. In other words, $\varphi$ induces a permutation on the set of components in $\mathcal S_1$ that preserves the complexity of the cycles, i.e. $\prec$-maximal minimal components in $\Gamma$ are mapped by $\varphi$ to $\prec$-maximal minimal components in $\Gamma^e$ and $\varphi$ preserves complexity of cycles from components in $\mathcal S_1$. To simplify the notation, without loss of generality, we shall assume that $g=1$. 

\smallskip

\textit{Step of induction.} To prove the induction step, we proceed by induction on the complexity of the cycles $d$ that belong to $S_q$ but not to $S_{q-1}$, denoted $d \in S_q \setminus S_{q-1}$, to show that the image of $S_q$ under $\varphi$ belongs to $\Gamma$.

Let $d$ be a cycle in $S_q\smallsetminus S_{q-1}$ of minimal length and maximal complexity. By induction hypothesis, we have that $\varphi(S_{q-1})=S_{q-1}' \in \mathcal S_{q-1}$, $S_{q-1}'\subseteq \Gamma< \Gamma^e$. Since $\varphi$ is a graph monomorphism  $\varphi(d) \cap S_{q-1}' = \varphi(d \cap S_{q-1})$ and so $\varphi(d)$ belongs to the neighbourhood of $S_{q-1}'$. Since $d$ has minimal length in $S_q$, it follows from Claim 3 that $\varphi(d)=d'$ belongs to $\Gamma<\Gamma^e$ and so $d'$ belongs $N(S_{q-1}')\cap \Gamma =S_q'$. 

Furthermore, since $S_q$ is $\prec$-maximal in $\mathcal T_q$, $\varphi(d)\subseteq \Gamma$, $\varphi$ is an embedding and since $d$ has maximal complexity in $S_q\smallsetminus S_{q-1}$, it follows that $d'$ has maximal complexity in $N(S_{q-1}')\cap \Gamma$.  We conclude that the bijection $\varphi$ on $S_{q-1}$ extends to a bijection between cycles of minimal length in  $S_q \setminus S_{q-1}$ and $S'_q \setminus S_{q-1}'$ and that cycles of maximal complexity in  $S_q \setminus S_{q-1}$ are mapped to cycles of maximal complexity in $S'_q \setminus S_{q-1}'$. Hence, all cycles of minimal length in the neighbourhood of $\varphi(S_{q-1})$ belong to the image of $\varphi(S_q)$. This proves the base of induction.

Assume by induction that $\varphi$ is a bijection between $S_{q-1}$ and $S_{q-1}'$ as well as between the cycles in $S_q\smallsetminus S_{q-1}$ and $S_q'\smallsetminus S_{q-1}'$ of length less than or equal to $l-1$ and cycles of length $l$ and complexity greater than $K$. Let $c$ be a cycle in $S_q\smallsetminus S_{q-1}$ of length $l$ and maximal complexity less than $K$. It follows from Claim 2, that either $\varphi(c) = c'\in S_q'\smallsetminus S_{q-1}'$ and $\comp(c)=\comp(c')$ or $\varphi(c)$ is contained in the image of shorter cycles and one of them belongs to the neighbourhood of $S_{q-1}'$. However, by induction, all the cycles in $S_q'\smallsetminus S_{q-1}'$ of length less than $l$ or of length $l$ and complexity greater than $K$ are in one-to-one correspondence with the cycles of the same complexity in $S_q\smallsetminus S_{q-1}$ and so in particular, they belong to the image of $\varphi(S_q)$. This implies that the $\varphi(c)$ is not contained in the image of shorter cycles and so we deduce that $\varphi(c) = c'\in S_q'\smallsetminus S_{q-1}'$. Moreover, since $S_q$ is $\prec$-maximal in $\mathcal T_q$, $\varphi(c)\in S_q'\smallsetminus S_{q-1}'$, $\varphi$ is an embedding and since $c$ has maximal complexity (in $S_q$) less than $K$ among all cycles in $S_q\smallsetminus S_{q-1}$, we conclude that $\comp(c') = \comp(c)$ and so we get a bijection between the cycles in $S_q\smallsetminus S_{q-1}$ and $S_q'\smallsetminus S_{q-1}'$ of length $l$ and maximal complexity less than $K$. This finishes the proof of the induction step.

The statement follows since $S_q=\Gamma$ for all large enough $q$.
\end{proof}

\begin{thm} \label{thm:atomic}
Let $\Delta$ and $\Gamma$ be two atomic graphs. Then $\GG(\Delta)$ and $\GG(\Gamma)$ embed into each other if and only if $\Delta=\Gamma$.
\end{thm}
\begin{proof}
Since atomic graphs are triangle-free, it follows from \cite{KK} that $\GG(\Delta) < \GG(\Gamma)$  (resp. $\GG(\Gamma) < \GG(\Delta)$) if and only if there is a graph embedding $\varphi:\Delta < \Gamma^e$ (resp. $\psi:\Delta < \Gamma^e$). 

Let $V(\Delta)=\{a_1,\dots, a_l\}$ and $V(\Gamma)=\{b_1,\dots, b_n\}$. Let 
$$
\varphi: a_i\mapsto b_{f(i)}^{w_i(b_1,\dots, b_k)}
$$ 
and 
$$
\psi: b_j\mapsto a_{g(j)}^{v_j(a_1,\dots,a_l)},
$$
where $w_i\in \GG(\Gamma)$, $v_j\in \GG(\Delta)$, $i=1,\dots, l$, $j=1,\dots, n$. We observe that $\varphi$ and $\psi$ induce homomorphisms $\bar \varphi:\GG(\Delta)\to \GG(\Gamma)$ and $\bar \psi:\GG(\Gamma)\to \GG(\Delta)$. By \cite{KK}, there exists $N\in \BN$ depending only on $\Gamma$ and $\Delta$ so that the homomorphisms $\varphi^*$ and $\psi^*$ induced by the maps 
$$
a_i\mapsto {(b_{f(i)}^{N})}^{w_i} \hbox{ and } b_j\mapsto {(a_{g(j)}^{N})}^{v_j}
$$
correspondingly, are group monomorphisms. 

Furthermore, the maps $\bar \varphi$ and $\bar \psi$ naturally induce embeddings $\varphi':\Delta^e\to \Gamma^e$ and $\psi':\Gamma^e\to \Delta^e$ as follows. In order to define $\varphi'$ and $\psi'$ it suffices to determine the images of the vertices of the extension graph. By definition of the extension graph, its vertices are labelled by $a_i^{u_i}$ and $b_j^{q_j}$, correspondingly. Set
$$
\varphi': a_i^{u_i}\mapsto b_{f(i)}^{\varphi^*(u_i)} \quad \hbox{and} \quad \psi': b_j^{q_j}\mapsto a_{g(j)}^{\psi^*(q_j)}.
$$
Note that $\varphi'$ and $\psi'$ are graph embeddings, since  so are $\varphi$ and $\psi$ and since $\varphi^*$ and $\psi^*$ are group embeddings. 

By Theorem \ref{lem:atomicembedding}, the embedding $\varphi\psi'$ of $\Delta$ into $\Delta^e$ is unique up to conjugacy and graph automorphism. It follows that if $i_1\ne i_2$, then $a_{g(f(i_1))}\ne a_{g(f(i_2))}$. Hence, we have that $f(i_1)\ne f(i_2)$ and we conclude that $|V(\Delta)|\le |V(\Gamma)|$. Moreover, since $\varphi$ is a graph embedding, it follows that if $(a_{i_1}, a_{i_2})$ is an edge of $\Delta$, then $(b_{f(i_1)}, b_{f(i_2)})$ is an edge of $\Gamma$. We conclude that $\Delta$ is a subgraph of $\Gamma$. 

An analogous argument for $\psi\varphi'$ shows that $\Gamma$ is a subgraph of $\Delta$. Therefore, $\Delta$ and $\Gamma$ are isomorphic graphs.
\end{proof}

\bigskip

Often, the quasi-isometric rigidity of the group is closely related with the group being co-Hopfian - a property which can be viewed as some type of algebraic rigidity. For instance, one of the consequences of Mostow rigidity is that irreducible lattices in semisimple Lie groups are co-Hopfian (with the exception of free groups).  Recall that a group is termed co-Hopfian if it satisfies the following equivalent conditions:
\begin{itemize}
\item It is not isomorphic to any proper subgroup;
\item Every injective endomorphism of the group is an automorphism.
\end{itemize} 

Besides this classical example, there are many other interesting co-Hopfian groups: non-elementary freely indecomposable torsion-free word-hyperbolic groups; the fundamental groups of finite volume pinched negatively curved manifolds of dimension greater than 2; closed aspherical manifolds $M$ with $i(M) \ne 0$, where $i$ is a homotopy invariant of closed manifolds which is multiplicative under finite covers (e.g. Euler characteristic, signature, simplicial volume, or $L^2$-Betti number); finitely generated torsion-free nilpotent groups with a specific automorphism group of their Lie algebra, etc (see \cite{B} and references there). Our next goal is to show that the quasi-isometric rigidity of the class of atomic pc groups is also related to the co-Hopfian property of the $\BQ$-completion of these groups.

\medskip

Let us first show that injective endomorphisms of an atomic pc group are rigid, that is, any embedding $\varphi: \GG(\Gamma) \to \GG(\Gamma)$ is the identity up to an automorphism and taking powers of the generators. 

\begin{cor}\label{cor:injendos}
Let $\Gamma$ be an atomic graph and $\psi: \GG(\Gamma) \to \GG(\Gamma)$ an injective endomorphism. Then, there exist $g\in \GG(\Gamma)$, $\sigma \in \Aut(\Gamma)$ and $k_v \in \BZ \setminus {0}$ such that for all $v\in V(\Gamma)$, we have that
$$
\psi(v) = g^{-1}  \sigma(v)^{k_v} g, v\in V(\Gamma).
$$
In other words, $\psi$ is, up to conjugacy, graph automorphism and taking powers, the identity endomorphism.
\end{cor}
\begin{proof}
Let $v\in V(\Gamma)$. Since $\Gamma$ is an atomic graph, it follows that the centraliser $C(v)$ of $v$ is isomorphic to $\BZ \times F$, where $F$ is a free group of rank greater than or equal to $2$. On the other hand, it follows from the description of centralisers in pc groups, that if $C(h)$ is isomorphic to $\BZ \times F$, where $F$ is a free group of rank greater than or equal to $2$, $h\in \GG(\Gamma)$, then $h$ is conjugate of a power of a generator, that is $h = (v^k)^g$, for $k \in \BZ$, $g\in \GG(\Gamma)$. Therefore, for all $v\in V(\Gamma)$, we have that $\psi(v) = (w_v^{k_v})^{g_v}$, where $w_v \in V(\Gamma)$, $k_v \in \BZ$ and $g_v \in \GG(\Gamma)$. Since $\psi$ is an embedding, we have that $\{ w_v^{g_v} \mid v\in V(\Gamma) \}$, viewed as vertices of $\Gamma^e$, span an induced subgraph isomorphic to $\Gamma$ and so the map $v \to w_v^{g_v}$ induces a graph embedding $\varphi$ from $\Gamma$ to the extension graph $\Gamma^e$. By Theorem \ref{lem:atomicembedding}, up to conjugacy and graph automorphism there is only one graph embedding induced by the identity map and so $w_v^{g_v}= w_v^{g}$ for all $v \in V(\Gamma)$ and $v \to w_v$ induces a graph automorphism $\sigma \in \Aut(\Gamma)$. We conclude that
$$
\psi(v) = g^{-1}  \sigma(v)^{k_v} g, v\in V(\Gamma)
$$
for some $g\in \GG(\Gamma)$, $\sigma \in \Aut(\Gamma)$ and $k_v \in \BZ \setminus {0}$.
\end{proof}

In the context of the curve complex of a surface, this type of rigidity for embeddings from pc groups to modular groups was proven by Aramayona and Souto in \cite{AS}, see also \cite{AL}. More precisely, the authors show that for some rigid finite sets $X$ of the curve complex $\mathcal C(S)$ of a surface $S \ne S_{1,2}$, every injective homomorphism from the pc group $\GG(X)$ into the modular group $Mod^{\pm}(S)$ is obtained, up to conjugation, by taking powers of roots of Dehn twists in the vertices of $X$.

\smallskip

Corollary \ref{cor:injendos} shows that the only reason why atomic pc groups are not co-Hopfian is that the injective endomorphisms induced by the maps $v\to v^{k_v}$ that send generators to proper powers of themselves are not automorphisms. We now pass to a divisible extension of the atomic pc group, its $\BQ$-completion, to assure that these injective endomorphisms induce automorphisms there and show that the $\BQ$-completions of atomic pc groups are co-Hopfian. 

\medskip

One of the classical theorems in the theory of abelian groups asserts that every abelian group can be embedded into a divisible abelian group. An analogous result for torsion-free locally nilpotent groups was proven by Mal'cev. Since then, many mathematicians such as Kontorivich, Hall and Baumslag, have studied the classes of groups for which there always exists the $n$-th root of an element as well as groups for which such $n$-th root is unique. In this context, and further developing ideas of Lyndon, Miasnikov and Remeslennikov showed that free groups also embed into a divisible group - the free $\BQ$-group, and described its algebraic structure, see \cite{MR}.

In the same spirit, we prove in \cite{CDK2} that every pc group $\GG(\Gamma)$ embeds into a divisible group and that the category of divisible $\GG(\Gamma)$-groups has an initial object, the $\BQ$-completion $\GG(\Gamma)^{\BQ}$ of $\GG(\Gamma)$. Furthermore, as in the case of free groups, the $\BQ$-completion $\GG(\Gamma)^{\BQ}$ can be described algebraically as an iterated sequence of extensions of centralisers of elements. 

Since the construction of the $\BQ$-completion is technically involved, for readers not familiar with this notion, we consider an intermediate group, denoted by $\mathcal G(\Gamma, \BQ)$, which is a subgroup of the $\BQ$-completion $\GG(\Gamma)^\BQ$ and the smallest extension of $\GG(\Gamma)$ which is co-Hopfian (when the graph $\Gamma$ is atomic). Given a simplicial graph $\Gamma$, we define the group $\mathcal G(\Gamma, \BQ)$ as the graph product with underlying graph $\Gamma$ and vertex groups isomorphic to $\BQ$. Note that we view $\BQ$ as a divisible abelian group, that is, in exponential notation, we have that for all $n\in \BN$ and for all $x\in \BQ$, there exists $y\in \BQ$ such that $y^n=x$. Moreover, we identify elements $q$ from the vertex group $\BQ$ associated to $v$ with $v^q$ and this way, the \emph{ring} $\BQ$ has a natural action on the vertex groups: given $p\in \BQ$ (viewed as a ring) and $v^q \in \BQ$ (an element of the vertex group), we define the action of $p$ on $v^q$ as $v^{pq}$. The defined group is not divisible, as only products of pair-wise commuting elements of the vertex groups have $n$-th roots for all $n \in \BN$. It is not difficult to see that $\GG(\Gamma) < \mathcal{G}(\Gamma, \BQ) < \GG(\Gamma)^{\BQ}$.

\begin{cor}
Let $\Gamma$ be an atomic graph. Then $\mathcal G(\Gamma, \BQ)$ {\rm(}resp. the $\BQ$-completion $\GG(\Gamma)^\BQ$ of $\GG(\Gamma)${\rm)} is co-Hopfian.
\end{cor}  
\begin{proof}
It suffices to notice that if the centraliser $C(h)$ is non-abelian if and only if it is isomorphic to $\BQ \times F$, where $F$ is the free product $\BQ^{(1)} \ast \cdots \ast \BQ^{(n)}$, $n \ge 2$ if and only if $h$ is a conjugate of an element of a vertex group, that is $h=g^{-1} v^q g$, where $v^q$ belongs to the vertex group $v$, $q\in \BQ$ and $g\in \mathcal G(\Gamma, \BQ)$.

For the $\BQ$-completion $\GG(\Gamma)^{\BQ}$, we have that the centraliser $C(h)$ is non-abelian if and only if $C(h)\simeq \BQ \times F^{\BQ}$, where $F^{\BQ}$ is the $\BQ$-completion of a (non-abelian) free group if and only if $h$ is a conjugate of a $\BQ$-power of a generator, that is $h= g^{-1} {(v^q)} g$, where $v \in V(\Gamma)$, $q\in \BQ$, and $g\in \GG(\Gamma)^{\BQ}$.

Now proof is analogous to the proof of Corollary \ref{cor:injendos}.
\end{proof}

Note that for arbitrary graphs, the groups $\mathcal{G}(\Gamma, \BQ)$ and $\GG(\Gamma)^{\BQ}$ are far from being co-Hopfian. It suffices to take $\Gamma$ to be the edgeless graph with more than one vertex or, for a connected example, take $\Gamma$ to be a path of length greater than 2.

\smallskip

We can re-formulate these observations in yet another way, as follows.

\begin{cor}
Let $\Gamma$ be an atomic graph. Then each injective endomorphism $\psi: \GG(\Gamma) \to \GG(\Gamma)$ extends to automorphisms $\psi'$  of $\mathcal G(\Gamma, \BQ)$ and $\psi''$ of $\GG(\Gamma)^{\BQ}$, i.e. one has the following commutative diagram:
$$
\begin{CD}
\GG(\Gamma)    @>>> \mathcal G(\Gamma, \BQ) @>>> \GG(\Gamma)^{\BQ} \\
@V{\psi}VV       @V{\psi'}VV        @V{\psi''}VV\\
\GG(\Gamma)  @>>> \mathcal G(\Gamma, \BQ) @>>> \GG(\Gamma)^{\BQ} 
\end{CD}
$$
\end{cor}

\section{Concluding remarks}

The main motivation that brought us to formulate Conjecture \ref{conj:qi->EGE} comes, in fact, from a stronger statement about asymptotic cones of pc groups. The study of asymptotic geometry of pc groups and, in particular, the structure of their asymptotic cones, led us to believe that asymptotic cones of pc groups can be classified up to bilipschitz equivalence in terms of their extension graphs. More precisely, we raise the following question.

\begin{question}\label{ques:ac}
Given simplicial graphs $\Gamma$ and $\Delta$, is it true that the asympotic cones of $\GG(\Gamma)$ and $\GG(\Delta)$ are bilipschitz equivalent if and only if $\Gamma < \Delta^e$ and $\Delta < \Gamma^e$?
\end{question}

Since asymptotic cones of quasi-isometric pc groups are bilipschitz equivalent, beside being extremely interesting on its own right, Conjecture \ref{conj:qi->EGE} can be viewed as a weaker formulation and an ideal test-case of Question \ref{ques:ac}.

\bigskip

As we mentioned in the introduction, Question \ref{ques:qi->emb} has a positive answer when we restrict to commensurable pc groups. Recall that two groups $G$ and $G'$ are (abstractly) commensurable if they have isomorphic finite index subgroups, that is there exists $H<_{fi} G$ and $H'<_{fi} G'$ such that $H\simeq H'$. In particular, groups that are commensurable are quasi-isometric. Although the following result is easy to prove, it is worth mentioning and we record it as a lemma.

\begin{lemma}\label{lem:commensurable}
If the pc groups $\GG(\Delta)$ and $\GG(\Gamma)$ are commensurable, then $\GG(\Gamma)< \GG(\Delta)$ and $\GG(\Delta)< \GG(\Gamma)$.
\end{lemma}
\begin{proof}
Assume that $\GG(\Delta)$ and $\GG(\Gamma)$ are commensurable. By definition, there exist finite index subgroups $H <_{fi} \GG(\Gamma)$, $K<_{fi} \GG(\Delta)$ such that $H \simeq K$. Since the subgroups $H$ and $K$ are of finite index, there exists $N\in \BN$ so that the subgroup generated by the $N$-th powers of the generators of $\GG(\Delta)$ (resp. of $\GG(\Gamma)$) is a subgroup of $H$ (resp. of $K$): 
$$
\begin{array}{l}
\langle x_1^N, \dots, x_n^N \mid x_i \in V(\Gamma) \rangle < H\\
\langle y_1^N, \dots, y_m^N \mid y_i \in V(\Delta) \rangle < K. %
\end{array}
$$

By \cite{K}, the subgroup $\langle x_1^N, \dots, x_n^N \mid x_i \in V(\Gamma) \rangle$ is isomorphic to  $\GG(\Gamma)$ and $\langle y_1^N, \dots, y_m^N \mid y_i \in V(\Delta) \rangle$ is isomorphic to $\GG(\Delta)$. Using these isomorphisms we have that $\GG(\Delta) \hookrightarrow H \simeq K < \GG(\Gamma)$ and vice-versa. Therefore, if $\GG(\Gamma)$ is commensurable to $\GG(\Delta)$, then $\GG(\Gamma)< \GG(\Delta)$ and $\GG(\Delta)< \GG(\Gamma)$.
\end{proof}

\smallskip

Since Question \ref{ques:qi->emb} has a positive answer for commensurable pc groups, one may wonder why we state Conjecture \ref{conj:qi->EGE} using embeddability in the extension graph, rather than just embeddability of groups as in Question \ref{ques:qi->emb}. The motivation comes from the study of pc groups $\GG(\Gamma)$ and $\GG(\Delta)$ for which there is an embedding of $\GG(\Delta)$ to $\GG(\Gamma)$ but no graph embeddings of $\Delta$ to the extension graph $\Gamma^e$, see \cite{CDK}. The nature of such embeddings indicates that the corresponding groups $\GG(\Delta)$ and $\GG(\Gamma)$ are not commensurable (in fact, they seem to be not quasi-isometric). This brought us to believe that if Question \ref{ques:qi->emb} has a positive answer, then so does Conjecture \ref{conj:qi->EGE}. Therefore, a good starting point to check if Conjecture \ref{conj:qi->EGE} is indeed a consequence of Question \ref{ques:qi->emb} would be to answer the following question, which is also of an independent interest.

\begin{question}\label{quest:commensurability}
If $\GG(\Delta)$ and $\GG(\Gamma)$ are commensurable {\rm(}and so $\GG(\Delta)< \GG(\Gamma)$ and $\GG(\Gamma)< \GG(\Delta)${\rm)}, does it follow that $\Delta <\Gamma^e$ and $\Gamma < \Delta^e$?
\end{question}

\smallskip

In some cases, for instance if the graph $\Gamma$ (or $\Delta$) is triangle-free, triangle-built, a tree or the complement of a tree, the embeddability between pc groups is equivalent to the graph embeddability in to the extension graph, that is $\GG(\Delta) < \GG(\Gamma)$ and $\GG(\Gamma)<\GG(\Delta)$ if and only if $\Delta < \Gamma^e$ and $\Gamma < \Delta^e$, see \cite{KK}, \cite{CDK} and \cite{C}. Therefore, in these cases, Question \ref{quest:commensurability} has a positive answer. Moreover, if one further assumes $\Gamma$ to be a triangle– and square-free graph without any degree-one or degree-zero vertex, then if $\GG(\Delta)$ and $\GG(\Gamma)$ are commensurable, we have that the corresponding extension graphs are isomorphic, i.e. $\Delta^e \simeq \Gamma^e$, see \cite[Proposition 7]{KK2}.

\bigskip

In the cases of quasi-isometric classification of pc groups that we analyse, we show not only that Conjecture \ref{conj:qi->EGE} holds, but also so does its converse. However, we would like to stress that graph embeddability into the extension graph is not a sufficient condition to assure quasi-isometry of the corresponding groups. One obvious obstruction is that the graphs $\Gamma$ and $\Delta$ need to be both either connected or disconnected since the number of ends is a quasi-isometry invariant, see \cite{P}. Similarly, since the cyclic JSJ is also a quasi-isometry invariant, one can find connected examples for which the converse of Conjecture \ref{conj:qi->EGE} does not hold: consider the connected union of two cycles $C_5$ of length 5, where the cycles only share a vertex $p$ and denote this graph by $C_5 \vee_p C_5$. It is easy to show that $C_5 \vee_p C_5$ embeds into the extension graph of $C_5$ and vice-versa. However, the corresponding pc groups are not quasi-isometric since $C_5 \vee_p C_5$ has a non-trivial JSJ-decomposition over the infinite cyclic group generated by $p$, while $C_5$ does not, see \cite{Clay}.

\medskip

Finally, we also mentioned that a positive answer to Conjecture \ref{conj:qi->EGE} would imply that coherence is a quasi-isometric invariant in the class of pc groups: if the pc group $\GG(\Gamma)$ is coherent and the pc group $\GG(\Delta)$ is quasi-isometric to $\GG(\Gamma)$, then $\GG(\Delta)$ is coherent. This is a consequence of the following remark.

\begin{remark}\label{rem:coherent}
In \cite{DromsCoherence}, Droms gives a graph theoretical characterisation of the class of coherent pc groups: a pc group $\GG(\Gamma)$ is coherent if and only if $\Gamma$ is chordal, i.e. $\Gamma$ does not contain cycles of length greater than or equal to $4$. If $\Gamma$ is chordal, then, as shown in \cite{KK}, the extension graph $\Gamma^e$ and any induced subgraph of $\Gamma^e$ are also chordal.  
\end{remark}


\begin{thebibliography}{10}

\bibitem[A]{Ahl} A. R. Ahlin. {\em The large scale geometry of products of trees}, Geom. Dedicata, {\bf 92} (2002), 179--184, Dedicated to John Stallings on the occasion of his 65th birthday.

\bibitem[AL]{AL} J. Aramayona, C. R. Leininger, {\em Finite rigid sets in curve complexes}, J. Topol. Anal. {\bf 5} (2013), no. 2, 183--203.

\bibitem[AS]{AS} J. Aramayona, J. Souto, {\em A remark on homomorphisms from right angled Artin groups to mapping class groups}, C. R. Acad. Sci. Paris {\bf 351} (2013), no. 19-20, 713--717.

\bibitem[BJN]{BJN} J. Behrstock, T. Januszkiewicz, W. Neumann, {\em Quasi-isometric classification of some high dimensional right-angled Artin groups}, Groups Geom. Dynam., {\bf 4} (2010), 681--692. 

\bibitem[BN]{BN} J. Behrstock, W. Neumann, {\em Quasi-isometric classification of graph manifolds groups}, Duke Math. J., {\bf 141} (2008), 217--240.

\bibitem[B]{B} I. Belegradek, {\em On co-Hopfian nilpotent groups}, Bull. London Math. Soc., {\bf 35}, (2003), 805--811.

\bibitem[BKS]{BKS} M. Bestvina, B. Kleiner, M. Sageev \emph{The asymptotic geometry of right-angled Artin groups, I} Geom. Topol. \textbf{12} (2008), 1653--1699.


\bibitem[C]{C} M. Casals-Ruiz, \emph{Embeddability and universal theory of partially commutative groups}, Internat. Math. Res. Notices 2015, doi: 10.1093/imrn/rnv122

\bibitem[CDK]{CDK} M. Casals-Ruiz, A. Duncan, I. Kazachkov, \emph{Embeddings between two partially commutative groups: two counterexamples}, Journal of Algebra, \textbf{390} (2013), 87--99.

\bibitem[CDK2]{CDK2} M. Casals-Ruiz, A. Duncan, I. Kazachkov, {\em Lyndon's completion for partially commutative groups}, preprint.

\bibitem[Cl]{Clay} M. Clay, \emph{When does a right-angled Artin group split over $\BZ$?}, arXiv:1403.1842v2

\bibitem[D1]{DromsIso} C.~Droms, {\em Isomorphisms of graph groups}, 
Proc. Amer. Math. Soc.,\textbf{100} (1987), no. 3, 407--408.

\bibitem[D2]{DromsCoherence} C. Droms, {\em Graph groups, coherence, and three-manifolds}, J. Algebra, {\bf 106} (1987), 484--489.

\bibitem[D3]{DromsPC} C. Droms, {\em Subgroups of graph groups}, J. Algebra, {\bf 110} (1987), no. 2, 519--522.

\bibitem[G]{Gromov} M. Gromov, {\em Groups of polynomial growth and expanding maps}, Inst. Hautes Études Sci. Publ. Math., {\bf 53} (1981), 53--73.

\bibitem[H]{Huang} J. Huang, {\em Quasi-isometry rigidity of right-angled Artin groups I: the finite out case}, arXiv:1410.8512

\bibitem[KKL]{KKL} M. Kapovich, B. Kleiner, B. Leeb, {\em Quasi-isometries and the de Rham decomposition} Topol. {\bf 37} (1998), no. 6, 1193--1211.

\bibitem[KK]{KK} S.-H. Kim, T. Koberda, {\em Embedability between right-angled Artin groups}, Geometry and Topology \textbf{17} (2013), 493--530.

\bibitem[KK2]{KK2}S.-H. Kim, T. Koberda,Intern. {\em  The geometry of the curve complex of a right-angled Artin group}, Internat. J. Algebra Comput. {\bf 24} (2014), no. 2, 121--169.

\bibitem[K]{K} T. Koberda \emph{Right-angled Artin groups and a generalized isomorphism problem for finitely generated subgroups of mapping class groups}, Geom. Funct. Anal. \textbf{22} (2012), no. 6, 1541--1590.

\bibitem[MSW]{MSW} L. Mosher, M. Sageev, K. Whyte. {\em Quasi-actions on trees I. Bounded valence} Ann. of Math., {\bf 158} (1), (2003), 115--164

\bibitem[MR]{MR} A.G. Myasnikov, V.N. Remeslennikov, {\em Exponential groups 2: extensions of centralisers and tensor completion of CSA-groups}, Int. J. Algebra Comput., {\bf 06}, (1996), 687--707.

\bibitem[P]{P} P. Papasoglou, {\em Quasi-isometry invariance of group splittings} Ann. Math., {\bf 161} (2005), 759--830

\bibitem[PW]{PW} P. Papasoglou, K. Whyte, {\em Quasi-isometries between groups with infinitely many ends}, Comment. Math. Helv., {\bf 77} (2002), 133--144.

\end{thebibliography}
\end{document}